\documentclass[a4paper, 10pt]{article}
\usepackage[utf8]{inputenc}
\usepackage{amsmath, amssymb, amsthm, bm, url}
\usepackage[textwidth = 14cm]{geometry}

\DeclareMathOperator{\adj}{adj}
\DeclareMathOperator{\rank}{rank}

\DeclareMathOperator{\conv}{conv}

\DeclareMathOperator{\supp}{supp}
\DeclareMathOperator{\intt}{int}
\DeclareMathOperator{\sign}{sign}
\DeclareMathOperator{\diam}{diam}

\newcommand{\R}{\mathbb{R}}
\newcommand{\Z}{\mathbb{Z}}
\newcommand{\N}{\mathbb{N}}
\newcommand{\Q}{\mathbb{Q}}

\newcommand{\bA}{\bm{A}}
\newcommand{\bx}{\bm{x}}

\newcommand{\bh}{\bm{h}}
\newcommand{\bB}{\bm{B}}
\newcommand{\ba}{\bm{a}}
\newcommand{\br}{\bm{r}}
\newcommand{\be}{\bm{e}}

\newcommand{\by}{\bm{y}}
\newcommand{\bz}{\bm{z}}

\newcommand{\bQ}{\bm{Q}}
\newcommand{\bb}{\bm{b}}
\newcommand{\bd}{\bm{d}}
\newcommand{\bv}{\bm{v}}
\newcommand{\bw}{\bm{w}}
\newcommand{\bU}{\bm{U}}
\newcommand{\bV}{\bm{V}}
\newcommand{\bW}{\bm{W}}
\newcommand{\bH}{\bm{H}}
\newcommand{\bT}{\bm{T}}

\newtheorem{theorem}{Theorem}
\newtheorem{proposition}[theorem]{Proposition}

\newtheorem{lemma}[theorem]{Lemma}
\newtheorem{corollary}[theorem]{Corollary}

\numberwithin{theorem}{section}

\title{On lattice width of lattice-free polyhedra and height of Hilbert bases}

\author{Martin Henk\thanks{Institut f\"{u}r Mathematik, Technische Universit\"{a}t Berlin, Germany, henk@math.tu-berlin.de} \and Stefan Kuhlmann\thanks{Institut f\"{u}r Mathematik, Technische Universit\"{a}t Berlin, Germany, kuhlmann@math.tu-berlin.de} \and Robert Weismantel\thanks{Department of Mathematics, Institute for Operations Research, ETH Z\"{u}rich, Switzerland, robert.weismantel@ifor.math.ethz.ch \newline The last two authors have received funding from the third authors Einstein Visiting Fellowship (project 1-4001022) issued by the Einstein Foundation}}

\begin{document}
	\maketitle
	\noindent \textbf{Abstract.} We study the lattice width of lattice-free polyhedra given by $\bA\bx\leq\bb$ in terms of $\Delta(\bA)$, the maximal $n\times n$ minor in absolute value of $\bA\in\Z^{m\times n}$. Our main contribution is to link the lattice width of lattice-free polyhedra to the height of Hilbert bases and to the diameter of finite abelian groups. This leads to a bound on the lattice width of lattice-free pyramids which solely depends on $\Delta(\bA)$ provided a conjecture regarding the height of Hilbert bases holds. Further, we exploit a combination of techniques to obtain novel bounds on the lattice width of simplices.\\
	A second part of the paper is devoted to a study of the above mentioned Hilbert basis conjecture. We give a complete characterization of the Hilbert basis if $\Delta(\bA) = 2$ which implies the conjecture in that case and prove its validity for simplicial cones.
	
	\section{Introduction}
	Given a polyhedron
	\begin{align*}
		P(\bA,\bb) := \lbrace \bx \in \R^n : \bA\bx\leq\bb\rbrace
	\end{align*}
	defined by $\bA\in\Z^{m\times n}$ with full column rank, $m\geq n$ and $\bb\in\Z^m$. The polyhedron is \textit{lattice-free} if $P(\bA,\bb)\cap \Z^n = \emptyset$ and the $\textit{lattice width}$ of $P(\bA,\bb)$ is defined as
	\begin{align*}
		w(P(\bA,\bb)) := \min_{\bz\in \Z^n\backslash\lbrace\bm{0}\rbrace}(\max_{\bx\in P(\bA,\bb)}\bz^T\bx - \min_{\by\in P(\bA,\bb)}\bz^T\by).
	\end{align*}
	More generally, given a set $X \subseteq \Z^n\backslash\lbrace \bm{0}\rbrace$ the lattice width of $P(\bA,\bb)$ in directions $X$ is
	\begin{align*}
		w^X(P(\bA,\bb)) := \min_{\bz\in X}(\max_{\bx\in P(\bA,\bb)}\bz^T\bx - \min_{\by\in P(\bA,\bb)}\bz^T\by).
	\end{align*}
	Let $\mathcal{F}$ be the set of facet normals of $P(\bA,\bb)$, which are a subset of the rows of $\bA$. Then $w^\mathcal{F}(P(\bA,\bb))$ denotes the \textit{facet width} of $P(\bA,\bb)$. In Subsection \ref{ssrelationlwidfwid} we present an upper bound on the facet width in terms of the lattice width, see Proposition \ref{proprelationwidth}.\\
	There are various upper bounds on the lattice width of lattice-free polyhedra which depend on the dimension $n$, e.g., the lattice width is $\mathcal{O}(n\log(n))$ if the number of vertices or facets of $P(\bA,\bb)$ can be bounded by a polynomial in $n$, see \cite{banalitvakpajorszarekflatness99}. The current best upper bound on the lattice width of the more general class of convex bodies is $\mathcal{O}^*(n^{\frac{4}{3}})$ where $\mathcal{O}^*$ denotes that a polynomial in $\log n$ is omitted, see \cite{rudelsonlatticewidthbound2000}.\\ 
	The lattice width plays a key role for a variety of algorithms related to integer programs. One famous example is Lenstras approach to solve the feasibility question of integer linear programs \cite{lenstraintprogr83}.\\ 
	Our aim is to bound the lattice/facet width of lattice-free polyhedra solely in terms of $\Delta(\bA)$ where
	\begin{align*}
		\Delta(\bA) := \max\lbrace |\det \bB| : \bB \text{ is an } n \times n \text{ submatrix of } \bA\rbrace.
	\end{align*}
	For $\Delta(\bA) =  1 $ it is well-known that each vertex of $P(\bA,\bb)$ is integral. Further, each full-dimensional polyhedron contains an integer point if $\Delta(\bA) = 2$, see \cite[Theorem 1]{veselovchirkovbimodular09}. Hence, lattice-free polyhedra with $\Delta(\bA)=2$ are contained in a hyperplane. That already implies $w(P(\bA,\bb))= 0 = \Delta(\bA) - 2$.\\
	Simplices are so far the only class of polytopes that are known to be bounded solely by $\Delta(\bA)$, as proven by Gribanov and Veselov \cite{gribanovvesleovwidth16}. More precisely, they show
	\begin{align}
		\label{gribanovsimplexbound}
		w^\mathcal{F}(P(\bA,\bb)) < \delta(\bA) - 1,
	\end{align}
	where $\bA\in\Z^{(n + 1)\times n}$ defines the simplex and
	\begin{align*}
		\delta(\bA) := \min\lbrace |\det \bB| : \bB \text{ is an } n \times n \text{ invertible submatrix of } \bA\rbrace.
	\end{align*}
	Furthermore, they present a bound on the lattice width of lattice-free polyhedra that depends among other things on $\Delta(\bA)$, the dimension and the lowest common multiple of all $n\times n$ minors, which can get exponentially large in $\Delta(\bA)$.\\
 	In a recent result, Basu and Jiang use the facet width of certain polytopes with bounded minors to construct an algorithm which efficiently enumerates special integer points inside those polytopes \cite{basujiangenumeratewidth21}. The latter result is one recent example among many research studies that try to establish a refined theory of integer optimization in which answers to questions do not only depend on the dimension and the number of rows of an underlying constraint matrix, but also involve the data parameter defined as the maximum absolute value among all subdeterminants of the constraint matrix. Besides the lattice width and facet width there are many other prominent functions of this kind: the diameter of a polyhedron, the support of optimal solutions to standard-form integer programs, the proximity of optimal integer and continuous solutions and the running time function for integer optimization problems. See \cite{alideloeisoerweissupportint2018,alievdeloesparelindio2017,arteisglanzoertvemweisipsmalldet2016,artweiszenbimodalgo2017,bonisummaeisenbranddiameterpoly14,eisenweissteinitz18,naegelesanzencongruence2021,naegelesudazensubmodminicongr2018,paatweiswelprox2018} for some references about this development.\\
	A core of our analysis is a new link between lattice-freeness, the height of a Hilbert basis and the diameter of finite abelian groups. This leads among others to a bound on the facet width of lattice-free pyramids, see Theorem \ref{facetwidthpyramidtheorem}.\\
	The diameter of a finite abelian group is defined as follows. Let $G$ be a finite (additive) abelian group and $H\subseteq G$ a set which generates $G$. Then the \textit{diameter of $G$ with respect to $H$} denotes the minimal $k\in \N$ such that $H$ generates $G$ with at most $k$ sums, i.e.,
	\begin{align*}
		\diam_H(G) := \min\lbrace k\in \N : \lbrace h^1+...+h^l : h^i\in H \text{ and } l\leq k\rbrace =G\rbrace.
	\end{align*}
	Further, the \textit{diameter of $G$} is given by
	\begin{align*}
		\diam(G) := \max\lbrace \diam_H(G) : H\subseteq G \text{ generates } G\rbrace
	\end{align*}
	and therefore it is independent of the generating set.\\
	As it turns out, the Hilbert basis of a specific cone forms a generating set of an appropriate finite abelian group that we will investigate in Section \ref{fwidpyramids}. Given a matrix $\bA\in\Z^{m\times n}$ with full column rank and $m\geq n$ we define the polyhedral cone
	\begin{align*}
		C(\bA) := \lbrace \bx\in\R^n : \bA\bx\geq 0\rbrace.
	\end{align*}
	So $C(\bA)$ is rational and pointed. The \textit{Hilbert basis elements} of such a cone $C(\bA)$ with respect to a lattice $\Lambda$ are the irreducible elements, i.e., $\bh\in C(\bA)\cap \Lambda$ is a Hilbert basis element if and only if for all $\bz^1, \bz^2 \in C(\bA)\cap \Lambda$ with $\bh = \bz^1 + \bz^2$ we have either $\bz^1 = \bm{0}$ or $\bz^2 = \bm{0}$. All elements in $C(\bA) \cap \Lambda$ can be described as a non-negative integral combination of the respective Hilbert basis elements. The set of all Hilbert basis elements is called the \textit{Hilbert basis}.\\
	In order to bound the facet width, we use a bound on each Hilbert basis element. This relates directly to the height of Hilbert bases.\\
	Each Hilbert basis element $\bh$ with respect to a lattice $\Lambda$ is a positive combination of the vectors which lie on the extreme rays of $C(\bA)$. If we denote the set of those vectors by $R$, then the \textit{height} of a Hilbert basis element equals 
	\begin{align*}
		\mathcal{H}_{R,\Lambda}(\bh) := \min \left\{\sum_{\br \in R}\lambda_{\br} : \sum_{\br \in R}\lambda_{\br}\br = \bh, \lambda_{\br}\geq 0 \text{ for all } \br\in R\right\}
	\end{align*}
	and the height of a Hilbert basis is 
	\begin{align*}
		\mathcal{H}_{R,\Lambda}(C(\bA)) := \max \lbrace \mathcal{H}_{R,\Lambda}(\bh) : \bh\in C(\bA) \cap \Lambda \text{ is a Hilbert basis element}\rbrace.
	\end{align*} 
	Commonly, the vectors $R$ are scaled to be primitive lattice vectors. In that case, the height is strictly bounded by $n-1$, see \cite{liutrotterzieglerheighthilbert93}. Some improvements of this bound with respect to determinants of these primitive vectors are given in \cite{henkweisminimalhilb97}. However, for our approach we need another scaling of the vectors which we define below.\\
	Let $\br\in C(\bA)$ lie on an extreme ray with
	\begin{align*}
		\bA_{I,\cdot}\br = \bm{0}
	\end{align*}
	for $I\subseteq \lbrack m\rbrack$ and $|I|=n-1$ such that $\rank(\bA_{I,\cdot}) = n - 1$, see beginning of Subsection \ref{ssnotation} for a definition of the matrices $\bA_{I,\cdot}$ and $\bA_{I,\lbrack n\rbrack\backslash \lbrace i\rbrace}$. Then by Cramer's rule we work with the following scaling (up to a sign)
	\begin{align*}
		r_i := (-1)^i\det \bA_{I,\lbrack n\rbrack\backslash \lbrace i\rbrace}
	\end{align*}
	for $i = 1,...,n$. This scaling could vary if there are other $n-1$ linearly independent rows of $\bA$ indexed by $\bar{I}\subseteq \lbrack m\rbrack$ with  
	\begin{align*}
		\bA_{\bar{I},\cdot}\br = \bm{0}
	\end{align*}
	and $\gcd(\bA_{\bar{I},\cdot})\neq \gcd(\bA_{I,\cdot})$, see (\ref{defgcd}) for a definition of this quantity. We scale as large as possible:
	\begin{align}
		\label{normgendefmax}
		\max_{I\subseteq\lbrack m\rbrack}\lbrace \gcd(\bA_{I,\cdot}) : I\subseteq\lbrack m\rbrack, \text{ }|I|=n-1, \text{ }\rank(\bA_{I,\cdot})= n -1,\text{ } \bA_{I,\cdot}\br = \bm{0}\rbrace.
	\end{align} 
	Throughout this paper we refer to these scaled vectors as \textit{normalized generators} and denote by $R(\bA)$ the set of normalized generators of $C(\bA)$. If the vectors are scaled to be primitive, we refer to them as \textit{primitive generators} and denote the set by $\tilde{R}(\bA)$.\\
	The scaling allows us to deduce the following fundamental property of the right-hand side $\bA\br$. Pick $k\in\lbrack m \rbrack$ and $\br\in R(\bA)$, we observe 
	\begin{align*}
		\bA_{k,\cdot}\br = |\det \bA_{I\cup\lbrace k\rbrace,\cdot}| \leq \Delta(\bA)
	\end{align*}
	by Laplace expansion. In particular, we get
	\begin{align}
		\label{normgenrhs}
		\Vert \bA\br\Vert_\infty\leq \Delta(\bA).
	\end{align}
	Note that $\br$ is not necessarily primitive. Therefore, the height of a Hilbert basis with respect to the normalized generators might be significantly smaller than $n -1$ and might not even depend on the dimension at all. A recent result by Sissokho \cite{papageometryminimalsolution21} shows that the height of a Hilbert basis with respect to the normalized generators is bounded by 1 if $m = n + 1$. Due to this and our following collection of results (Theorem \ref{theoremsumupSHC}), we conjecture the following:\\\\
	\noindent
	\textbf{Strong Hilbert basis conjecture (SHC).} Let $C(\bA)$ be a cone, $\bh\in C(\bA)$ a Hilbert basis element of $C(\bA)$ with respect to $\Z^n$ and $R(\bA) = \lbrace \br^1,...,\br^t\rbrace$. Then 
	\begin{align*}
		\bh\in \conv\lbrace \bm{0},\br^1,...,\br^t\rbrace
	\end{align*}
	or equivalently $\mathcal{H}_{R(\bA),\Z^n}(C(\bA)) \leq 1$.\\\\
	\noindent
	Prior to this paper, the conjecture was known to be true for $m = n + 1$. It is also true if we fix $\Delta(\bA) = 1$, see e.g., \cite[Proposition 8.1]{Sturmfels1995GrobnerBA}. More specifically, the Hilbert basis elements are the primitive vectors on each extreme ray in that case. These results combined with our upcoming analysis (Theorem \ref{bimodtheorem}, Proposition \ref{transformationprop} and Corollary \ref{simplicialstatement}) yield the following cases where (SHC) holds.
	\begin{theorem}
		\label{theoremsumupSHC}
		(SHC) holds if at least one of the following conditions is satisfied:
		\begin{enumerate}
			\item $m\leq n + 1$
			\item all $n\times n$ minors of $\bA$ are contained in $\lbrace 0,\pm k, \pm 2k\rbrace$ for some $k\in\N_{\geq 1}$.
		\end{enumerate}
	\end{theorem}
	The proof of Theorem \ref{theoremsumupSHC} is presented in Subsection \ref{ssectionsumupproof}.\\
	We emphasize that the scaling in (\ref{normgendefmax}) is important. There are counterexamples to the conjecture if we use a different scaling, see Appendix \ref{appendixexample1} for an example.\\
	Throughout the paper we work with cones defined by $\bA\bx\geq \bm{0}$. Nevertheless, we can reformulate the setting of (SHC) in terms of the orthogonal subspace of the image of $\bA$. For that purpose let $\bW\in\Z^{(m - n)\times m}$ be a full row rank matrix with $\bW\bA=\bm{0}$. Then in the orthogonal formulation we work with the cone defined by $\bW\bx = \bm{0}$ and $\bx\geq \bm{0}$. Further, the normalized generators correspond to suitably scaled circuits. This orthogonal setting is equivalent to ours and builds the foundation of \cite{papageometryminimalsolution21}.\\ 
	In order to state our facet width, let us make precise what a pyramid means. Let $\bv$ satisfy $\bA\bv=\bb$ for a full column rank matrix $\bA\in \Z^{m\times n}$ with $m\geq n$ and $\bb\in\Z^m$. One might think of $\bv$ as the \textit{apex}. We investigate the translated cone 
	\begin{align*}
		\bv + \lbrace \bx\in \R^n : -\bA \bx \geq 0\rbrace = \bv - C(\bA)
	\end{align*}
	and some $-\ba \in \intt(C(\bA)^*)\cap \Z^n$, where $C(\bA)^*$ is the dual cone of $C(\bA)$, see (\ref{defdualcone}) for a precise definition. The \textit{pyramid} is then defined by 
	\begin{align*}
		P(\bA,\ba,\bb,b_a):=\left\{ \bx\in \R^n : \begin{pmatrix}
			\bA \\ \ba^T
		\end{pmatrix}\bx\leq\begin{pmatrix}
			\bb \\ b_a
		\end{pmatrix}\right\}
	\end{align*}
	where $b_a$ is an integer greater than $\ba^T\bv$. Observe that $-\ba \in \intt(C^*)$ guarantees the boundedness of $P(\bA,\ba,\bb,b_a)$. We remark that the arguments in this paper stay valid if $-\ba$ lies on the boundary of $C(\bA)^*\cap \Z^n\backslash \lbrace \bm{0}\rbrace$. In that case, $P(\bA,\ba,\bb,b_a)$ is unbounded.\\
	Furthermore, we refer to a polyhedron $P(\bA,\bb)$ as \textit{$\Delta$-modular} if $\Delta(\bA) = \Delta$. Finally, in order to state our facet width bound, we denote by $\Gamma^*$ the dual lattice of $\Gamma$, see (\ref{defduallattice}) for a definition and more about lattices and their bases.
	\begin{theorem}
		\label{facetwidthpyramidtheorem}
		Let $P=P(\bA,\ba,\bb,b_a)$ be a lattice-free $\Delta$-modular pyramid. Further, let $\Lambda=(\bA^T\Z^m)^*$ and $\bB\in \Q^{n\times n}$ be a basis of $\Lambda$. Then 
		\begin{align}
			\label{pyramidthmres}
			w^{\ba}(P)\leq\frac{\Delta}{\gcd(\bA)}\mathcal{H}_{R(\bA\bB),\Z^n}(C(\bA\bB))\diam(\Lambda / \Z^n) - 1
		\end{align}
		and if (SHC) is true, then
		\begin{align}
			\label{pyramidthmresSHC}
			w^{\ba}(P)\leq \Delta - 2.
		\end{align}
	\end{theorem} 
	The second inequality is based on a result of Klopsch and Lev \cite{klopschlevgenabeliangroups09} about the diameter of finite abelian groups. Furthermore, we present examples where the first bound in (\ref{pyramidthmres}) is tight for arbitrary finite abelian groups.\\
	We emphasize that proving any constant upper bound on $\mathcal{H}_{R(\bA\bB),\Z^n}(C(\bA\bB))$ implies directly an upper bound in (\ref{pyramidthmres}) which depends linearly on $\Delta$.\\
	Moreover, recall that $\tilde{R}(\bA\bB)$ denotes the set of primitive generators of $C(\bA\bB)$. We obtain
	\begin{align*}
		\mathcal{H}_{R(\bA\bB),\Z^n}(C(\bA\bB))\leq \mathcal{H}_{\tilde{R}(\bA\bB),\Z^n}(C(\bA\bB)) < n -1
	\end{align*}
	using the bound in \cite{liutrotterzieglerheighthilbert93}. This leads to
	\begin{align*}
		w^{\ba}(P)< (n - 1)(\Delta - 1) - 1.
	\end{align*}
	As a consequence of Theorem \ref{facetwidthpyramidtheorem}, we are able to generalize (\ref{gribanovsimplexbound}) by exploiting the underlying group structure. We abbreviate
	\begin{align*}
		\delta := \delta\left(\begin{pmatrix}
			\bA \\ \ba^T
		\end{pmatrix}\right).
	\end{align*}
	\begin{corollary}
		\label{corfacetwidthsimplex}
		Let $P=P(\bA,\ba,\bb,b_a)$ be a lattice-free simplex with $|\det \bA|=\delta$. Then 
		\begin{align*}
			w^\mathcal{F}(P) < \diam(\Lambda / \Z^n)\leq \delta - 1. 
		\end{align*}
	\end{corollary}
	This is not a direct improvement of (\ref{gribanovsimplexbound}), but the first inequality solely depends on the group structure of $\Lambda / \Z^n$ and is usually significantly smaller than $\delta - 1$. Hence, this bound is more accurate.\\
	We apply Corollary \ref{corfacetwidthsimplex} to improve (\ref{gribanovsimplexbound}) in terms of the lattice width.
	\begin{theorem}
		\label{latticewidthsimplextheorem}
		Let $P=P(\bA,\ba,\bb,b_a)$ be a lattice-free simplex with $|\det \bA|=\delta$. Then
		\begin{align*}
			w(P) < \left\lfloor \frac{\delta}{2}\right\rfloor.
		\end{align*}
	\end{theorem}
	This is the current best upper bound on the lattice width of simplices with respect to their minors. We prove all the lattice/facet width results in Section \ref{fwidpyramids}.\\
	As mentioned earlier, the lattice width of lattice-free pyramids is closely tied to (SHC). Therefore, we study in Section \ref{sectionSHCspecialcases} the conjecture and prove its validity for some classes of cones. In particular, we characterize the Hilbert basis if $\Delta(\bA) = 2$. Throughout the paper we refer to $2$-modular polyhedra as \textit{bimodular}. Furthermore, if $\bh$ lies on an extreme ray, we say that $\bh$ is \textit{trivial}. Otherwise we call $\bh$ \textit{non-trivial}.
	\begin{theorem}\label{bimodtheorem}
		Let $C(\bA)$ be a bimodular cone and $\bh$ a non-trivial Hilbert basis element of $C(\bA)$ with respect to $\Z^n$. Then, there exist $\br^i,\br^j\in R(\bA)$ with
		\begin{align*}
			\bh = \frac{1}{2}\br^i+\frac{1}{2}\br^j
		\end{align*}
		for $i\neq j$.
	\end{theorem}
	Theorem \ref{bimodtheorem} implies that (SHC) holds in the bimodular case. While proving this result we rediscover that all Hilbert basis elements of $C(\bA)$ are trivial for $\Delta(\bA) = 1$.\\
	We also show that (SHC) behaves well under certain linear transformations.
	\begin{proposition}
		\label{transformationprop}
		Let $C(\bA)$ be a cone and $\bB\in\Z^{n\times n}$ be an invertible matrix. If (SHC) holds for $C(\bA)$, then it is true for $C(\bA\bB)$. Furthermore, it suffices to prove (SHC) when $\gcd(\bA) = 1$.
	\end{proposition}
	As a result, we are able to give an affirmative answer to (SHC) in the simplicial case and locate the Hilbert basis elements in terms of lower-dimensional faces. More specifically, we prove the following.
	\begin{corollary}\label{simplicialstatement}
		Let $C(\bA)$ be a simplicial cone and $\bh$ a Hilbert basis element of $C(\bA)$ with respect to $\Z^n$. Then, $\mathcal{H}_{R(\bA),\Z^n}(C(\bA))\leq 1$. Furthermore, $\bh$ lies in the interior of a $k$-face of $C(\bA)$ with $k\leq \Delta$.
	\end{corollary}
	Combining this with \cite{papageometryminimalsolution21} implies that (SHC) holds if $m\leq n + 1$. Moreover, Corollary \ref{simplicialstatement} shows (SHC) for $n = 2$ as every cone in dimension two is simplicial. Already for simplicial cones there are examples where the height of a Hilbert basis with respect to the primitive generators asymptotically attains $n - 1$, compare with \cite{henkweisminimalhilb97} and \cite{liutrotterzieglerheighthilbert93}. This and Corollary \ref{simplicialstatement} highlight the discrepancy between primitive and normalized generators. We prove Proposition \ref{transformationprop} and Corollary \ref{simplicialstatement} in Subsection \ref{ssectionsimplicial}.\\
	Theorem \ref{bimodtheorem} enables us to identify classes of pyramids where our bound (\ref{pyramidthmresSHC}) holds.
	\begin{corollary}
		\label{corbimodandpyramid}
		The bound (\ref{pyramidthmresSHC}) holds if at least one of the following is satisfied:
		\begin{enumerate}
			\item $\Delta(\bA)$ is prime. 
			\item All $n\times n$ minors of $\bA$ are in $\lbrace 0, \pm\frac{\Delta(\bA)}{2}, \pm\Delta(\bA)\rbrace$.
		\end{enumerate}
	\end{corollary}
	We prove this statement in Section \ref{fwidpyramids}.\\
	We remark that the lattice width cannot lead to bounds which solely depend on $\Delta(\bA)$ if $\intt(P(\bA,\bb))\cap \Z^n = \emptyset$. A detailed example for that is given in (\ref{exampledependdimension}).
	
	\subsection{Notation and definitions}\label{ssnotation}
	Here, we introduce the notation and definitions which are used throughout the paper. We abbreviate $\lbrack m\rbrack :=\lbrace 1,...,m\rbrace$. Given $\bA\in\Z^{m\times n}$, $I\subseteq\lbrack m \rbrack$ and $J\subseteq \lbrack n\rbrack$ then $\bA_{I,J}$ denotes the submatrix of $\bA$ with rows indexed by $I$ and columns indexed by $J$. When $J = \lbrack n\rbrack$, we use $\bA_{I,\cdot}$ and vice versa. In case of $|I| = |J|$, we write $\bA_{I,J}^{-1}$ for $(\bA_{I,J})^{-1}$. Further, $\adj(\bA_{I,J})$ denotes the adjugate matrix of $\bA_{I,J}$. If $m\geq n$, we use
	\begin{align}
		\label{defgcd}
		\gcd(\bA) := \gcd(\det \bA_{I,\cdot} : I\subseteq\lbrack m\rbrack \text{ with } |I|= n)
	\end{align}
	and set $\gcd(\bA^T) := \gcd(\bA)$. Additionally, $\rank(\bA)$ denotes the (column) rank of $\bA$. For $\bx\in\R^n$ we represent \textit{the support of $\bx$} as
	\begin{align*}
		\supp(\bx) := \lbrace i\in\lbrack n\rbrack : x_i\neq 0\rbrace.
	\end{align*}
	An $\bx\in \Z^n$ is called \textit{primitive} if $\gcd(\bx) = 1$.
	If not defined otherwise, we denote by $C(\bA)$ the cone 
	\begin{align*}
		C(\bA) := \lbrace \bx\in \R^n : \bA\bx\geq\bm{0}\rbrace
	\end{align*}
	with $\bA\in\Z^{m\times n}$ and $\rank(\bA)=n$. The cone $C(\bA)$ is \textit{simplicial} if $m=n$. We denote by
	\begin{align}
		\label{defdualcone}
		C(\bA)^* := \lbrace \bx \in \R^n : \by^T\bx\geq 0 \text{ for all }\by\in C(\bA)\rbrace
	\end{align}
	the \textit{dual cone} of $C(\bA)$. An \textit{extreme ray} is a 1-dimensional face of $C(\bA)$. Let 
	\begin{align*}
		P(\bA,\bb) := \lbrace \bx\in \R^n : \bA\bx\leq \bb \rbrace
	\end{align*}
	be a polytope with $\bA\in\Z^{m\times n}$ for $m\geq n$ and $\rank(\bA)=n$ and $\bv\in P(\bA,\bb)$ a vertex of $P(\bA,\bb)$. Further, let $I\subseteq \lbrack m\rbrack$ be the maximal index set with $\bA_{I,\cdot}\bv = \bb_I$. Then
	\begin{align*}
		C^{\bv} := P(\bA_{I,\cdot},\bb_I) = \lbrace \bx\in \R^n : \bA_{I,\cdot}\bx \leq \bb_I \rbrace
	\end{align*}
	is the \textit{vertex cone} of $\bv$.\\
	A \textit{lattice} $\Lambda$ is a discrete subgroup of $\R^n$. A \textit{basis} of $\Lambda$ are $k$ linearly independent vectors $\bb^1,...,\bb^k\in \Lambda$ such that $\Lambda := (\bb^1,...,\bb^k)\Z^k$. If $k = n$, the lattice is full-dimensional. The \textit{dual lattice} $\Lambda^*$ of a full-dimensional lattice $\Lambda$ is defined by
	\begin{align}
		\label{defduallattice}
		\Lambda^* := \lbrace \bx\in \R^n : \by^T\bx\in \Z \text{ for all }\by\in\Lambda\rbrace.
	\end{align}
	The canonical unit vectors of $\R^n$ are denoted by $\be_1,...,\be_n$. The vector $\bm{1}\in \R^n$ represents the all-ones vector in $\R^n$ and $\bm{I}_n$ is the $n\times n$ unit matrix.\\ 
	We abbreviate $\Delta(\bA)$ by $\Delta$ and $\delta(\bA)$ by $\delta$ if the underlying constraint matrix is clear from the context.\\
	$GL(n,\Z)$ denotes the group of all $n\times n$ unimodular matrices, i.e., $\bA \in\Z^{n \times n}$ and $|\det \bA| = 1$, $\intt(X)$ for $X\subseteq \R^n$ is the interior of $X$ and $\N= \lbrace 0,1,2,...\rbrace$.
	
	\subsection{Relation between lattice and facet width} \label{ssrelationlwidfwid}
	In this subsection we show that there is a natural relation between lattice and facet width in terms of the $n\times n$ minors of the given constraint matrix.\\
	Let $P = P(\bA,\bb)$ be a polytope defined by a full column rank matrix $\bA\in\Z^{m\times n}$ for $m\geq n$. In contrast to the rest of the paper, we allow the right-hand side to be real, i.e., $\bb\in \R^m$. Further, we assume that all rows of $\bA$ define a facet of $P$. Therefore, the set of facet normals $\mathcal{F}$ corresponds to the rows of $\bA$.\\
	Given any $\ba\in \mathcal{F}$ we get by definition the following trivial bounds
	\begin{align*}
		w(P)\leq w^\mathcal{F}(P)\leq w^{\ba}(P).
	\end{align*}
	A converse relationship is presented next.
	\begin{proposition}\label{proprelationwidth}
		Let $P = P(\bA,\bb)$ be a $\Delta$-modular polytope. Then
		\begin{align*}
			w^\mathcal{F}(P)\leq \Delta w(P).
		\end{align*}
	\end{proposition}
	\begin{proof}
		Let $\ba\in\Z^n\backslash\lbrace \bm{0}\rbrace$ be a flat direction attaining $w(P)$. Since $P$ is bounded, we can pick vertices $\bv,\bw\in P$ that maximize/minimize $\ba^T\bx$ over $P$. Thus, 
		\begin{align*}
			w(P) = \ba^T(\bv-\bw).
		\end{align*}
		By Carath\'eodory's theorem there exist $n$ linearly independent rows of $\bA$ which define the vertex cone $C^{\bv}$, say $\bA_{1,\cdot},...,\bA_{n,\cdot}$, with
		\begin{align}
			\label{lemmarelwidth1}
			\ba^T = \sum_{i = 1}^n \lambda_i\bA_{i,\cdot}
		\end{align}
		where we assume that $\lambda_1,...,\lambda_k$ are positive for $k\leq n$ and $\lambda_i = 0$ if $i \in \lbrack n\rbrack\backslash\lbrack k\rbrack$. We select a vertex $\tilde{\bw}$ which maximizes $-\bA_{1,\cdot}\bx$. Then 
		\begin{align*}
			w(P) & =  \ba^T\bv-\ba^T\bw \\
			& \geq \ba^T\bv-\ba^T\tilde{\bw} \\ 
			& =  \sum_{i = 1}^k \lambda_i(\bA_{i,\cdot}\bv - \bA_{i,\cdot}\tilde{\bw}) \\ 
			& =  \sum_{i = 1}^k \lambda_i (b_i - \bA_{i,\cdot}\tilde{\bw}) \\
			& = \lambda_1 w^{\bA_{1,\cdot}^T}(P) + \sum_{i = 2}^k \lambda_i (b_i - \bA_{i,\cdot}\tilde{\bw}) \\
			&\geq \lambda_1 w^{\bA_{1,\cdot}^T}(P).
		\end{align*}
		As $\lambda_1$ is positive and $\ba\in \Z^n$, we get $\lambda_1\geq \Delta^{-1}$ by Cramer's rule applied to (\ref{lemmarelwidth1}). The claim follows from $w^\mathcal{F}(P)\leq w^{\bA_{1,\cdot}^T}(P)$.
	\end{proof}
	Suppose there is a function $f$ where $f$ solely depends on $\Delta$ such that $w(P)\leq f(\Delta)$. Then the proof of Proposition \ref{proprelationwidth} shows us that $w^\mathcal{F}(P)\leq \Delta f(\Delta)$.\\ 
	If we focus on the case $\Delta = 1$, we obtain the following interesting corollary.
	\begin{corollary}
		\label{cordelta=1rel}
		Let $P=P(\bA,\bb)$ be a polytope with $\Delta = 1$. Then
		\begin{align*}
			w(P)=w^{\mathcal{F}}(P).
		\end{align*}
	\end{corollary}
	We close this subsection by providing an example $P(\bA,\bb)$ where $\intt(P(\bA,\bb))\cap\Z^n = \emptyset$ and the lattice width depends on the dimension.\\
	Let
	\begin{align}
		\label{exampledependdimension}
		S = \left\{ \bx\in \R^n : \begin{pmatrix} -1 & 0 & \hdots & 0 \\
			0 & -1  &  & \\
			\vdots & & \ddots &  \\
			0 & & & -1 \\
			1 & 1 & \hdots & 1
		\end{pmatrix}\bx \leq \begin{pmatrix}
			0 \\
			0 \\
			\vdots \\
			0 \\
			n
		\end{pmatrix}\right\}
	\end{align}
	be a simplex scaled with the dimension $n$. It is easy to see that $\intt(S)\cap\Z^n = \emptyset$ and $\Delta = 1$. So Corollary \ref{cordelta=1rel} tells us that it suffices to calculate $w^{\mathcal{F}}(S)$ in order to obtain $w(S)$. Doing so yields $n=w^{\mathcal{F}}(S)=w(S)$.
	
	\section{Facet width of lattice-free pyramids} \label{fwidpyramids}
	In this section we present a link between (SHC), the diameter of finite abelian groups and the lattice width of lattice-free pyramids in facet direction.\\
	We want to understand the lattice-freeness of the pyramid $P(\bA,\ba,\bb,b_a)$. As a first step, we fix the vertex cone $C^{\bv} = \bv - C(\bA)$ at the apex $\bv$. Our aim is to determine how far we can push out the facet defined by $\ba$ until we are guaranteed to hit an integer point. In other words, how big does $b_a$ need to be such that we can guarantee $P(\bA,\ba,\bb,b_a)\cap \Z^n \neq \emptyset$.\\
	In order to tackle this question, we utilize some group structure coming from the integrality of the constraint matrix and the right-hand side of $P(\bA,\ba,\bb,b_a)$. Let 
	\begin{align*}
		\Lambda = \lbrace \bx\in \R^n : \bA\bx\in \Z^m\rbrace
	\end{align*}
	be the lattice of integral right-hand sides for $\bA$. Note that 
	\begin{align}
		\label{definitionlattice}
		\Lambda = (\bA^T\Z^m)^*
	\end{align}
	as each element $\bx\in (\bA^T\Z^m)^*$ satisfies by our definition of the dual lattice $\be_i^T\bA\bx\in \Z$ for all $i\in\lbrack m \rbrack$ and every element in $\by\in\Lambda$ satisfies $\bz^T\bA\by\in \Z$ for all $\bz\in \Z^m$.\\
	Observe that $\bv\in \Lambda$ and $\Z^n\subseteq \Lambda$. Let $\bh^1,...,\bh^t\in C(\bA) \cap \Lambda$ denote the Hilbert basis elements of $C(\bA)$ with respect to $\Lambda$. Then every element in $(\bv - C(\bA))\cap\Lambda$ can be written as
	\begin{align}
		\label{fwidexplain1}
		\bv - \sum_{i=1}^t\lambda_i\bh^i\in\Lambda
	\end{align}
	for integral non-negative $\lambda_1,...,\lambda_t$. Specifically, all elements in $(\bv - C(\bA)) \cap \Z^n$ are of the form above. Let an integral point of the form (\ref{fwidexplain1}) be given. Since $b_a\in\Z$, the lattice-freeness of $P(\bA,\ba,\bb,b_a)$ yields
	\begin{align*}
		\ba^T(\bv - \sum_{i=1}^t\lambda_i\bh^i)\geq b_a + 1.
	\end{align*}
	Hence, we obtain
	\begin{align*}
		w^{\ba}(P(\bA,\ba,\bb,b_a)) = b_a - \ba^T\bv \leq \sum_{i=1}^t\lambda_i(-\ba^T\bh^i) - 1.
	\end{align*}
	We proceed by constructing an upper bound on $\lambda_1+...+\lambda_t$ via the diameter of finite abelian groups. A second step is then to bound $-\ba^T\bh^i$ using (SHC).\\
	First, we pass to the finite abelian group $\Lambda / \Z^n$. Here, an integer point of the form (\ref{fwidexplain1}) corresponds an additive inverse of the coset $\bv+\Z^n$ in $\Lambda / \Z^n$. Let $H = \lbrace \bh^1...,\bh^t\rbrace / \Z^n$ be the set of Hilbert basis elements modulo $\Z^n$. As the Hilbert basis generates the lattice $\Lambda$, the set $H$ is a generating set of $\Lambda / \Z^n$. Thus, we need at most $\diam_H(\Lambda / \Z^n)$ sums to generate all elements, including an additive inverse of $\bv+\Z^n$ in $\Lambda / \Z^n$.\\
	We want to bound $\diam_H(\Lambda / \Z^n)$ from above. The definitions already yield $\diam_H(\Lambda / \Z^n)\leq \diam(\Lambda / \Z^n)$. \\
	By the fundamental theorem of finite abelian groups every finite abelian group can be decomposed into the direct sum of cyclic groups, i.e., let $G$ be a finite abelian group then there exist $s_1,...,s_N\in \N_{\geq 2}$ for $N\in\N_{\geq 1}$ such that
	\begin{align*}
		G \cong \Z / s_1\Z \oplus ... \oplus \Z / s_N\Z
	\end{align*}
	and $s_1|s_2|...|s_N$. We refer to this as the \textit{invariant decomposition} and the $s_i$ as the \textit{invariant factors}. In dependence of this decomposition there exists an exact formula for the diameter of a finite abelian group by Klopsch and Lev.
	\begin{theorem}[\cite{klopschlevgenabeliangroups09}, Theorem 2.1]
		\label{generatortheorem}
		Let $G$ be a finite abelian group with invariant decomposition $G\cong \Z /s_1 \Z \oplus ... \oplus \Z / s_N\Z$. Then
		\begin{align*}
			\diam(G)=\sum_{i=1}^Ns_i-1.
		\end{align*}
	\end{theorem}
	It is not hard to see that 
	\begin{align}\label{groupsumbound}
		\sum_{i=1}^N s_i-1\leq\left(\prod_{i=1}^{N}s_i\right)-1.
	\end{align}
	The product of the invariant factors equals the order of $G$. In our case the order of $G$ is the index of $\Lambda / \Z^n$. The following lemma quantifies this index which is given by $(\det \Lambda)^{-1}$.
	\begin{lemma}
		\label{latticedeterminant}
		Let $\bA\in\Z^{m\times n}$ with full column rank, $m\geq n$ and $\Lambda = (\bA^T\Z^m)^*$. Then 
		\begin{align*}
			\det \Lambda = \frac{1}{\gcd(\bA)}. 
		\end{align*}
	\end{lemma}
	\begin{proof}
		We prove $\det \bA^T\Z^m = \gcd(\bA)$. Then the claim follows from 
		\begin{align*}
			\det \bA^T\Z^m = \det \Lambda^* = (\det \Lambda)^{-1}.	
		\end{align*}
		We transform $\bA^T$ into Smith normal form, see e.g., \cite[Chapter 4.4]{schrijvertheorylinint86} for a treatment of Smith normal forms. In order to do so, there are $\bU \in GL(n,\Z)$ and $\bV \in GL(m,\Z)$ such that
		\begin{align*}
			\bU \bA^T \bV = \begin{pmatrix}
				\begin{matrix}
					\alpha_1 & & \\
					 & \ddots & \\
					 & & \alpha_n
				\end{matrix} \quad \bm{0}
			\end{pmatrix}
		\end{align*}
		with $\alpha_1 \cdot ...\cdot \alpha_n = \gcd(\bA)$. In the following, we apply $\bV\Z^m = \Z^m$, $\det \bU \Lambda^* = \det \Lambda^*$ and obtain
		\begin{align*}
			\det \bA^T\Z^m = \det \bU\bA^T \Z^m = \det \bU \bA^T \bV \Z^m = \alpha_1 \cdot ...\cdot \alpha_n = \gcd(\bA).
		\end{align*}
	\end{proof}
	Lemma \ref{latticedeterminant} allows us to bound $\diam_H(\Lambda / \Z^n)$. It remains to bound the scalar product $-\ba^T\bh^i$. Here we make use of (SHC).
	For sake of readability, we set
	\begin{align*}
		\Delta(\ba) := \Delta\left(\begin{pmatrix}
			\bA \\
			\ba^T
		\end{pmatrix}\right).
	\end{align*}
	In the following lemma, we allow $\ba$ to be rational. Note that the definition of $\Delta(\bA)$ extends naturally to rational matrices.
	\begin{lemma}
		\label{boundhilbertfacet}
		Given a cone $C=C(\bA)$ and $-\ba\in C^*\cap \Q^n$. Let $\bh\in C$ be a Hilbert basis element of $C$ with respect to $\Z^n$. Then 
		\begin{align*}
			-\ba^T\bh\leq \mathcal{H}_{R(\bA),\Z^n}(C)\Delta(\ba).
		\end{align*}
	\end{lemma}
	\begin{proof}
		Define 
		\begin{align*}
			Q := \bigcap_{-\ba\in C^*\cap \Q^n} \lbrace \bx\in\R^n : -\ba^T\bx\leq \mathcal{H}_{R(\bA),\Z^n}(C)\Delta(\ba)\rbrace \cap C
		\end{align*}
		the infinite intersection of all polyhedra which arise by intersecting the cone $C$ with one of the half-spaces defined by $-\ba\in C^*\cap \Q^n$. We show $\bh\in Q$.\\
		Let $R(\bA) = \lbrace \br^1,...,\br^t\rbrace$. Boundedness of the right-hand side, cf. (\ref{normgenrhs}), yields
		\begin{align*}
			-\ba^T\br^i\leq \Delta(\ba)
		\end{align*}
		for all $-\ba\in C^*\cap \Q^n$ and $i=1,...,t$. Thus, $\mathcal{H}_{R(\bA),\Z^n}(C)\br^i\in Q$. Since the intersection of infinitely many convex sets is convex, we conclude
		\begin{align}\label{proofweakhilbsets}
			\conv\lbrace \bm{0},\mathcal{H}_{R(\bA),\Z^n}(C)\br^1,...,\mathcal{H}_{R(\bA),\Z^n}(C)\br^t\rbrace \subseteq Q.
		\end{align}
	\end{proof}
	
	\begin{proof}[Proof of Theorem \ref{facetwidthpyramidtheorem}]
		Recall that $P=P(\bA,\ba,\bb,b_a)$. By definition we have $\bv\in\Lambda$, cf. (\ref{definitionlattice}). Note that $\bv - C(\bA)$ is the vertex cone of $P$ at $\bv$. Let $\bh^1,...,\bh^t$ denote the Hilbert basis of $C(\bA)$ with respect to $\Lambda$.\\
		Assume 
		\begin{align*}
			\bv - \sum_{i=1}^t\lambda_i\bh^i \in \Z^n
		\end{align*}
		for $\lambda_1,...,\lambda_t\in \Z_{\geq 0}$. Then the coset
		\begin{align*}
			- \sum_{i=1}^t\lambda_i\bh^i+\Z^n
		\end{align*}
		is an additive inverse of $\bv+\Z^n$ in $\Lambda / \Z^n$. So we can choose the coefficients to satisfy
		\begin{align}\label{prooffacetwidthbounddiam}
			\lambda_1+...+\lambda_t\leq \diam_H(\Lambda / \Z^n)\leq\diam(\Lambda / \Z^n)
		\end{align}
		where $H=\lbrace \bh^1,...,\bh^t\rbrace / \Z^n$. Since $P$ is lattice-free, we have 
		\begin{align}\label{prooffacetwidthcut}
			b_a + 1\leq \ba^T(\bv - \sum_{i=1}^t\lambda_i\bh^i).
		\end{align}
		Recall that $\bB\in\Q^{n\times n}$ is a basis of $\Lambda$. Then $\bB^{-1}C(\bA) = C(\bA\bB)$ and $\bB^{-1}\bh^1,...,\bB^{-1}\bh^t$ are Hilbert basis elements of $C(\bA\bB)$ with respect to $\Z^n$. Specifically, we get $\bA\bB\in\Z^{m\times n}$ by definition of $\Lambda$. Lemma \ref{latticedeterminant} implies 
		\begin{align*}
			\Delta\left(\begin{pmatrix}
				\bA \\ \ba^T
			\end{pmatrix}\bB\right)=\frac{\Delta(\ba)}{\gcd(\bA)}
		\end{align*}
		where $\Delta(\ba)=\Delta$. Combining this with Lemma \ref{boundhilbertfacet} and  $-\bB^T\ba\in \intt(C(\bA\bB)^*)\cap\Q^n$ results in
		\begin{align}
			\label{prooffacetwidthboundhilbert}
			-\ba^T\bh^i = -\ba^T\bB\bB^{-1}\bh^i \leq \frac{\Delta}{\gcd(\bA)}\mathcal{H}_{R(\bA\bB),\Z^n}(C(\bA\bB)).
		\end{align}
		As a consequence of (\ref{prooffacetwidthcut}), (\ref{prooffacetwidthboundhilbert}) and (\ref{prooffacetwidthbounddiam}), we obtain
		\begin{align*}
			w^{\ba}(P) &= b_a-\ba^T\bv \\
			&\leq \ba^T(\bv - \sum_{i=1}^t\lambda_i\bh^i) - 1 -\ba^T\bv \\
			&=-\sum_{i=1}^t\lambda_i\ba^T\bh^i - 1\\
			&\leq \frac{\Delta}{\gcd(\bA)}\mathcal{H}_{R(\bA\bB),\Z^n}(C(\bA\bB))\sum_{i=1}^t\lambda_i - 1\\
			&\leq \frac{\Delta}{\gcd(\bA)}\mathcal{H}_{R(\bA\bB),\Z^n}(C(\bA\bB))\diam(\Lambda / \Z^n) - 1.
		\end{align*}	
		For the second inequality in Theorem \ref{facetwidthpyramidtheorem}, we get $\mathcal{H}_{R(\bA\bB),\Z^n}(C(\bA\bB)) \leq 1$ from the validity of (SHC). Further, we note that 
		\begin{align*}
			\diam(\Lambda / \Z^n)\leq \gcd(\bA)-1
		\end{align*}
		from (\ref{groupsumbound}) and Lemma \ref{latticedeterminant}. As a result, we obtain
		\begin{align*}
			w^{\ba}(P)\leq \frac{\Delta}{\gcd(\bA)}\mathcal{H}_{R(\bA\bB),\Z^n}(C(\bA\bB))\diam_H(\Lambda / \Z^n) - 1\leq \Delta \frac{\gcd(\bA)-1}{\gcd(\bA)}-1\leq \Delta - 2.
		\end{align*}
		
	\end{proof}
	
	\begin{proof}[Proof of Corollary \ref{corbimodandpyramid}]
		We need (SHC) to be true for the cone $C(\bA\bB)$ in the proof above. Applying Theorem \ref{bimodtheorem} and the result for $\Delta(\bA\bB) = 1$ this is the case when
		\begin{align*}
			\frac{\Delta(\bA)}{\gcd(\bA)}\in\lbrace 1,2\rbrace
		\end{align*}
		which holds if the second statement from the claim is satisfied. Additionally, the first statement, $\Delta(\bA)$ being prime, implies $\gcd(\bA)=\Delta(\bA)$. That is why, the existence of a $n\times n$ minor of $\bA$ with determinant smaller than $\Delta(\bA)$ yields $\gcd(\bA) = 1$ which is equivalent to $\Lambda = \Z^n$. Hence, $\bv\in \Lambda$ contradicts the lattice-freeness of $P$.
	\end{proof}
	The second part of Theorem \ref{facetwidthpyramidtheorem} is true if (SHC) holds. A more detailed analysis reveals that (SHC) is required to be true for Lemma \ref{boundhilbertfacet} with $\mathcal{H}_{R(\bA\bB),\Z^n}(C(\bA\bB))\leq 1$. This still holds under the following Weak Hilbert basis conjecture.\\\\
	\noindent
	\textbf{Weak Hilbert basis conjecture.} Given a cone $C(\bA)$ and $-\ba\in C(\bA)^*\cap \Z^n$. Let $\bh\in C(\bA)$ be a Hilbert basis element with respect to $\Z^n$. Then 
	\begin{align*}
		-\ba^T\bh\leq \Delta(\ba).
	\end{align*}
	Since the rational case $\ba \in \Q^n$ follows directly from the integral case, we can restrict ourselves in the Weak conjecture to $\ba\in\Z^n$. Lemma \ref{boundhilbertfacet} shows that the Strong version implies the Weak version, hence the name. However, both conjectures are equivalent if $C$ is simplicial. That is why, $\conv\lbrace \bm{0},\br^1,...,\br^n\rbrace$ from (\ref{proofweakhilbsets}) is a simplex and the facet containing $\br^1,...,\br^n$ is defined by a facet normal in $C(\bA)^*\cap\Q^n$. Hence, we have equality in (\ref{proofweakhilbsets}).\\
	A special case of the pyramid where all vertex cones are simplicial is the simplex. In that case, we are able to show more. Theorem \ref{facetwidthpyramidtheorem} applied to the simplex case gives us the following quantities. We have $\Lambda = \bA^{-1}\Z^n$ and the Hilbert basis elements of $C(\bA)$ with respect to $\Lambda$ are the columns of $\bA^{-1}$. In particular, all Hilbert basis elements of $C(\bA)$ are trivial. Further, $\gcd(\bA)=\Delta(\bA)$ and the fact that $\bA^{-1}$ is a basis of $\Lambda$ imply that (SHC) needs to hold for the cone $C(\bm{I}_n)$. This is true by Corollary \ref{corbimodandpyramid}.
	\begin{proof}[Proof of Corollary \ref{corfacetwidthsimplex}]
		Recall that $P=P(\bA,\ba,\bb,b_a)$ is a simplex with $|\det \bA| = \delta$. Let $\bh^1,...,\bh^n$ be the Hilbert basis elements of $C(\bA)$ with respect to $\Lambda$. We denote by 
		\begin{align*}
			\bv,\bv-\lambda_1\bh^1,...,\bv-\lambda_n\bh^n
		\end{align*}
		the vertices of $P$ for some positive scalars $\lambda_1,...,\lambda_n$. We prove the corollary by contradiction. Assume that $w^\mathcal{F}(P) \geq \diam(\Lambda / \Z^n)$. This yields
		\begin{align*}
			\diam(\Lambda / \Z^n) & \leq w^{\bA_{i,\cdot}^T}(P) = \bA_{i,\cdot}\bv - \bA_{i,\cdot}(\bv - \lambda_i\bh^i) =\lambda_i
		\end{align*}
		for all $i=1,...,n$ since the Hilbert basis elements are columns of $\bA^{-1}$. The lower bound on the scalars implies
		\begin{align}
			\label{prooffacetwidthsimplex1}
			\bv - \underbrace{\conv\lbrace \bm{0},\diam(\Lambda / \Z^n)\bh^1,...,\diam(\Lambda / \Z^n)\bh^n\rbrace}_{=: M}\subseteq P.
		\end{align}
		We observe that
		\begin{align*}
			M = \left\{\bx\in \R^n : \begin{pmatrix}
				\bA \\ 
				-\bm{1}^T\bA
			\end{pmatrix}\bx\leq \begin{pmatrix}
			\bm{0} \\ 
			\diam(\Lambda / \Z^n)
		\end{pmatrix}\right\},
		\end{align*}
		i.e., $M$ is a simplex with all $n\times n$ minors being $\pm \det \bA$. This results in 
		\begin{align*}
			w^{-\bA^T\bm{1}}(\bv - M) = \diam(\Lambda / \Z^n).
		\end{align*}
		From Theorem \ref{facetwidthpyramidtheorem} it follows that there is an integer point in $\bv-M$. Further, the inclusion (\ref{prooffacetwidthsimplex1}) implies that $P$ contains an integer point, a contradiction to the lattice-freeness. 
	\end{proof}
	We finish this section by providing examples where the inequality (\ref{pyramidthmres}) of Theorem \ref{facetwidthpyramidtheorem} is tight for arbitrary finite abelian groups.\\
	Let
	\begin{align*}
		\Z / s_1\Z \oplus ... \oplus \Z / s_N\Z = G
	\end{align*}
	be an invariant decomposition. In dependence of $G$ we define the $N$-dimensional simplex
	\begin{align*}
		S^G=\left\{ \bx\in \R^N : \begin{pmatrix} s_1 & 0 & \hdots & 0 \\
			0 & s_2  &  & \\
			\vdots & & \ddots &  \\
			0 & & & s_N \\
			-s_1 & -s_2 & \hdots & -s_N
		\end{pmatrix}\bx \leq \begin{pmatrix}
		- 1 \\
		- 1 \\
		\vdots \\
		- 1 \\
		(\sum_{i=1}^N s_i) - 1
	\end{pmatrix}\right\}.
	\end{align*}
	Here the $N\times N$ diagonal matrix corresponds to $\bA$ and the last row to $\ba^T$ from Theorem \ref{facetwidthpyramidtheorem}. In particular, we have
	\begin{align*}
		\gcd(\bA) = \prod_{i = 1}^N s_i = \Delta.	
	\end{align*}
	Further, our vertex and Hilbert basis elements of $C$ with respect to $\Lambda$ have the following form
	\begin{align*}
		\bv = \begin{pmatrix}
			- \frac{1}{s_1} \\
			- \frac{1}{s_2} \\
			\vdots \\
			- \frac{1}{s_N}
		\end{pmatrix}, \bh^1 = \begin{pmatrix}
			\frac{1}{s_1} \\
			0 \\
			\vdots \\
			0
		\end{pmatrix},...,\bh^N = \begin{pmatrix}
			0 \\
			\vdots \\
			0 \\
			\frac{1}{s_N}
		\end{pmatrix}.
	\end{align*}
	By construction 
	\begin{align*}
		\bv - (s_1 - 1) \bh^1 - ... - (s_N - 1) \bh^N\in \Z^N	
	\end{align*}
	is the integer point which minimizes the coefficient sum. Note that
	\begin{align*}
		(\sum_{i=1}^N s_i) - 1 = \ba^T(\bv - (s_1 - 1) \bh^1 - ... - (s_N - 1) \bh^N) - 1.
	\end{align*}
	Hence, $S^G$ is lattice-free. The facet width in the direction of the last row is
	\begin{align*}
		w^{\ba}(S^G) = (\sum_{i=1}^N s_i) - 1 - \ba^T\bv = \diam(G) - 1.
	\end{align*}
	This follows from Theorem \ref{generatortheorem}.
	
	\subsection{Lattice width of lattice-free simplices}
	The diameter $\diam(\Lambda / \Z^n)$ in Corollary \ref{corfacetwidthsimplex} is usually significantly smaller than $\delta - 1$. We strengthen (\ref{groupsumbound}) by incorporating the number of cyclic groups in the invariant decomposition.
	\begin{lemma}
		\label{latticewidthlemmasumbound}
		Let $G = \Z /s_1 \Z \oplus ... \oplus \Z / s_N\Z$ be an invariant decomposition with $N\in \N_{\geq 1}$ and $\prod_{i=1}^Ns_i = \delta$. We have
		\begin{align*}
			\diam(G)\leq \left\lfloor \frac{\delta}{2^{N - 1}} + N - 2\right\rfloor.
		\end{align*}
	\end{lemma}
	\begin{proof}
		Theorem \ref{generatortheorem} states
		\begin{align}\label{prooflemmalatticewidth1}
			\diam(G)=\sum_{i=1}^N s_i - 1.
		\end{align}
		We prove the claim by induction on $N$. The group $G$ is cyclic if $N = 1$ and therefore $\diam(G) =  \delta - 1$.\\
		Let $N\geq 2$ and $s_1,...,s_N$ be the invariant factors which attain (\ref{prooflemmalatticewidth1}). Set
		\begin{align*}
			\tilde{G} = \Z /s_1 \Z \oplus ... \oplus \Z / s_{N-1}\Z.
		\end{align*}
		This is another invariant decomposition with 
		\begin{align*}
			\diam(\tilde{G})\leq \left\lfloor \frac{\delta}{2^{N - 2}s_N} + N - 3\right\rfloor
		\end{align*}
		by induction hypothesis. We get
		\begin{align*}
			\diam(G) =\sum_{i=1}^N s_i - 1 \leq \frac{\delta}{2^{N - 2}s_N} + N - 3 + s_N - 1 =:f(s_N)
		\end{align*}
		By calculating first and second derivatives of the function $f:\left[2,\frac{\delta}{2^{N-1}}\right]\to\R$, which is defined above, we observe that $f$ attains its maxima on the boundary. The claim follows from the integrality of $\diam(G)$.
	\end{proof}
	This bound is sharp if $\delta = 2^N$. However, in other cases it is still far away from being optimal. Nevertheless, it serves our purposes as we see in the proof of Theorem \ref{latticewidthsimplextheorem}.\\
	Furthermore, we analyze the relationship between the cardinality of the generating set and its diameter for cyclic groups. Another theorem from Klopsch and Lev \cite{klopschlevgenabeliangroups09} helps us again. They define
	\begin{align*}
		\phi_j(G)=\max \lbrace |H| : H\subseteq G, H \text{ generates } G, j\leq \diam_H(G)\rbrace
	\end{align*}
	for a finite abelian group $G$ and $j\in \N$. This quantity measures the maximal cardinality among all generating sets of $G$ which need at least $j$ sums to generate $G$. Moreover, the definition immediately implies $\phi_1(G)\geq \phi_2(G)\geq ... \geq \phi_{\diam(G)}(G)$. We have the following exact formula if $G$ is cyclic.
	\begin{theorem}[\cite{klopschlevgenabeliangroups09}, Theorem 2.5]\label{generatingcyclicbound}
		Let $\delta\in \N_{\geq 3}$ and $j\in \lbrack 2,\delta - 1\rbrack$. Then
		\begin{align*}
			\phi_j(\Z / \delta \Z) = \max \left\{ \frac{\delta}{d}\left( \left\lfloor \frac{d-2}{j - 1}\right\rfloor + 1\right) : d \textnormal{ divides }\delta, d\geq j + 1 \right\}.
		\end{align*}
	\end{theorem}
	This is all we need for the proof.
	\begin{proof}[Proof of Theorem \ref{latticewidthsimplextheorem}]
		Recall that $P=P(\bA,\ba,\bb,b_a)$ is a simplex with $|\det \bA| = \delta$. If $\delta = 1$, one of the vertices is integral by Cramer's rule. Further, if $\delta = 2$, the claim follows from Corollary \ref{corfacetwidthsimplex}. Thus, we assume $\delta\geq 3$.\\
		We set $\Lambda = \bA^{-1}\Z^n$ and denote by $\bh^1,...,\bh^n$ the columns of $\bA^{-1}$. Further, we write $H=\lbrace \bh^1,...,\bh^n\rbrace / \Z^n$. Note, Lemma \ref{latticewidthlemmasumbound} and Corollary \ref{corfacetwidthsimplex} imply that we may assume that $\Lambda / \Z^n$ is a cyclic group, i.e.,
		\begin{align*}
			\Lambda / \Z^n \cong \Z / \delta \Z.
		\end{align*}
		Further, let $k$ denote the number of different cosets in $H$. We want to identify the maximal $k$ such that we need strictly more than $\lfloor \frac{\delta}{2}\rfloor$ sums to generate the group. This can be done by the exact formula in Theorem \ref{generatingcyclicbound} which yields
		\begin{align}\label{prooflatwidsimplex1}
			\phi_{\lfloor \frac{\delta}{2}\rfloor + 1}(\Z / \delta \Z) = \left( \left\lfloor \frac{\delta-2}{\left\lfloor\frac{\delta}{2}\right\rfloor}\right\rfloor + 1\right) = 2.
		\end{align}
		Hence, for $k\geq 3$ we know that $\lfloor \frac{\delta}{2}\rfloor$ sums suffice to generate the group, i.e., $\diam_H(\Lambda / \Z^n)\leq\lfloor \frac{\delta}{2}\rfloor$. Therefore, assume that $k\leq 2$. Observe that $k = 2$ already implies that one of the two cosets is trivial. Otherwise we would be able to add the trivial coset to our generating set and obtain a set with cardinality three which contradicts the maximality in (\ref{prooflatwidsimplex1}).\\
		So we can assume that without loss of generality
		\begin{align*}
			\bh^1+\Z^n = ...= \bh^l+\Z^n	
		\end{align*}
		belong to the non-trivial coset for $l=2,...,n$. Hence, $\bh^{l+1}+\Z^n,...,\bh^n+\Z^n$ are trivial, i.e., $\bh^{l+1},...,\bh^n$ are integral. If $l=n$, then $H$ contains only one coset. We exploit this structure to construct a flat direction for these special simplices.\\
		Define the matrix
		\begin{align*}
			\tilde{\bA}^{-1} = (\bh^1,\bh^2 - \bh^1,...,\bh^l-\bh^1,\bh^{l+1},...,\bh^n).
		\end{align*}
		Every column apart from the first one is integral. Furthermore, the construction guarantees $|\det \bA^{-1}|=|\det \tilde{\bA}^{-1}|$. Let $\bT$ be the transformation matrix which scales the first column of $\tilde{\bA}^{-1}$ by $\delta$ and leaves the rest unchanged. Thus, $\tilde{\bA}^{-1}\bT \in \Z^{n\times n}$ and 
		\begin{align*}
			|\det \tilde{\bA}^{-1}\bT| = 1,
		\end{align*}
		so $\tilde{\bA}^{-1}\bT\in GL (n,\Z)$. Therefore, the inverse is unimodular, too, and we get
		\begin{align*}
			\bd = \frac{1}{\delta}(\bA_{1,\cdot}+...+\bA_{l,\cdot})= \be_1^T\bT^{-1}\tilde{\bA}
		\end{align*}
		is integral. We show that $\bd$ is a flat direction. For this purpose, let
		\begin{align*}
			\bv,\bv-\lambda_1 \bh^1,...,\bv - \lambda_n\bh^n
		\end{align*}
		be the vertices of $P$ with $\bv$ being the vertex tight at the inequalities defined by $\bA$ and $\lambda_1,...,\lambda_n$ positive scalars. By construction we have
		\begin{align*}
			\bd\bh^i = 0
		\end{align*}
		for $i = l + 1,...,n$. Without loss of generality let $\bv - \lambda_1\bh^1$ be the vertex which minimizes $\bd^T\bx$. Recall, that we have 
		\begin{align*}
			\lambda_1 = \bA_{1,\cdot}\bv - \bA_{1,\cdot}(\bv - \lambda_1\bh^1) = w^{\bA_{1,\cdot}^T}(P) < \delta - 1
		\end{align*}
		by Corollary \ref{corfacetwidthsimplex}. This yields
		\begin{align*}
			w^{\bd^T}(P)=\bd\bv - \bd(\bv - \lambda_1\bh^1) = \frac{\lambda_1}{\delta}< 1-\frac{1}{\delta}.
		\end{align*}
		In summary, all simplices with 
		\begin{align*}
			\diam(\Lambda / \Z^n)>\left\lfloor \frac{\delta}{2}\right\rfloor
		\end{align*}
		attain a flat direction with $w(P)<1$. Thus, the claim follows.
	\end{proof}
	
	\section{Towards a proof of (SHC)} \label{sectionSHCspecialcases}
	
	We begin our discussions by showing that it is possible to impose more structure on (SHC) without loss of generality.
	\subsection{Reduction to full-dimensional version}
	Let us introduce the following geometric object which plays a key role in proving Theorem \ref{bimodtheorem}. For $\by\in C(\bA)$ the \textit{spindle} is defined by
	\begin{align*}
		S(\by) := \lbrace \bx\in \R^n : \bm{0}\leq \bA\bx\leq \bA\by\rbrace.
	\end{align*}
	We can reformulate the irreducibility of Hilbert basis elements in terms of a spindle property. By definition we have $\bh\in C(\bA)$ is a Hilbert basis element with respect to a lattice $\Lambda$ if and only if 
	\begin{align}
		\label{spindleproperty}
		S(\bh)\cap \Lambda=\lbrace \bm{0},\bh\rbrace.
	\end{align}
	We refer to our problem as full-dimensional for a given Hilbert basis element $\bh\in C(\bA)$ if $\dim(S(\bh)) = n$ where $\dim(S(\bh))$ denotes the dimension of the linear space spanned by the elements in $S(\bh)$. This holds if and only if $\bh \in \intt(C(\bA))$.\\\\ \noindent
	\textbf{Full-dimensional Strong Hilbert basis conjecture.} Let $C(\bA)$ be a (full-dim\-ensional) cone, $\bh\in \intt(C(\bA))$ a Hilbert basis element of $C(\bA)$ with respect to $\Z^n$. Then 
	\begin{align*}
		\mathcal{H}_{R(\bA),\Z^n}(C(\bA))\leq 1.
	\end{align*}
	Furthermore, we can rewrite $\br\in R(\bA)$ with $\bA_{I,\cdot}\br = \bm{0}$ in terms of the adjugate matrix. Let $\bA_{j,\cdot}$ be a row with $\bA_{j,\cdot}\br > 0$. Then
	\begin{align}
		\label{alternativedefnormgen}
		\br = |\det \bA_{J,\cdot}|\bA_{J,\cdot}^{-1}\be_n=\sign(\det \bA_{J,\cdot})\adj(\bA_{J,\cdot})\be_n
	\end{align}
	for $J = I \cup \lbrace j\rbrace$ where without loss of generality the last row of $\bA_{J,\cdot}$ corresponds to the row indexed by $j$. Note that the representation does not depend on the choice of $\bA_{j,\cdot}$.\\
	Next, we reduce the general version to the full-dimensional one.
	\begin{proposition}\label{pointinterior}
		The full-dimensional conjecture implies the (SHC).
	\end{proposition}
	\begin{proof}[Proof of Proposition \ref{pointinterior}]
		Given a cone $C(\bA)$, which is not necessarily full-dimensional. Let $\bh\in C(\bA)$ be in the (relative) interior of a $k$-dimensional face of $C(\bA)$ with $1\leq k\leq n - 1$, i.e., there exists $I\subseteq\lbrack m\rbrack$ with $|I| = n - k$, $\bA_{I,\cdot}\bh =\bm{0}$, $\rank(\bA_{I,\cdot}) = n - k$ and $\gcd(\bA_{I,\cdot})$ maximal among all rows describing the $k$-dimensional face. We define by
		\begin{align*}
			F = C(\bA) \cap \lbrace \bx\in\R^n : \bA_{I,\cdot}\bx = \bm{0}\rbrace
		\end{align*}
		this $k$-face. Surely, $\bh$ is a Hilbert basis element of $F$ with respect to $\Z^n$ if and only if the same holds for the cone $C(\bA)$.\\
		In the following, we transform $F$ into coordinate hyperplanes, project $F$ onto the non-zero coordinates and check that the scaling works out.\\
		There exists a unimodular transformation $\bU\in GL(n,\Z)$, such that
		\begin{align*}
			\bA_{I,\cdot}\bU = (
			\bH,\bm{0})
		\end{align*}
		for some invertible matrix $\bH\in\Z^{(n-k)\times(n-k)}$, e.g., by transforming $\bA_{I,\cdot}$ into Hermite normal form. Moreover, one can show 
		\begin{align}
			\label{prooffulldimreduction1}
			|\det \bH| = \gcd(\bA_{I,\cdot})	
		\end{align}
		via the Smith normal form, compare with the proof of Lemma \ref{latticedeterminant}. Since (SHC) is invariant under unimodular transformations, we can work with the face 
		\begin{align*}
			\bU^{-1}F\subseteq \lbrace \bx\in \R^n : x_1=...=x_{n-k} = 0\rbrace.
		\end{align*}
		From now on, we assume that our face $F$ has the form from above, i.e., lies in the first $n-k$ coordinate hyperplanes.\\
		Let $\tilde{F}$ be the projection of $F$ onto the last $k$ coordinates. So $\tilde{F}$ is $k$-dimensional. As a result, we have
		\begin{align*}
			\tilde{F} = \lbrace \bx\in\R^k : \bA_{\cdot, \lbrack n \rbrack\backslash \lbrack n - k\rbrack}\bx\geq \bm{0}\rbrace = C(\bA_{\cdot, \lbrack n \rbrack\backslash \lbrack n - k\rbrack}).
		\end{align*}
		The cone $\tilde{F}$ is defined by the last $k$ columns of $\bA$. Further, each point $\bx\in F$ with $\bx^T=(\bm{0},\tilde{\bx}^T)$ corresponds one-to-one to $\tilde{\bx}\in \tilde{F}$. Hence, $\bh$ is a Hilbert basis element of $F$ with respect to $\Z^n$ if and only if $\tilde{\bh}$ is a Hilbert basis element of $\tilde{F}$ with respect to $\Z^k$. Furthermore, $\tilde{\bh}\in\intt(\tilde{F})$.\\
		It is left to check the scaling of our new normalized generators. Each normalized generator $\br$ of $F$ corresponds to a normalized generator $\tilde{\br}$ of $\tilde{F}$. Let 
		\begin{align*}
			\br = |\det \bA_{J,\cdot}|\bA_{J,\cdot}^{-1}\be_n=\sign(\det \bA_{J,\cdot})\adj(\bA_{J,\cdot})\be_n
		\end{align*}
		for some $J\subseteq\lbrack m\rbrack$.\\
		Suppose $I\subseteq J$ and let $\tilde{I}=J\backslash I$. This implies that the matrix $\bA_{J,\cdot}$ is up to row permutations of the following form
		\begin{align*}
			\bA_{J,\cdot} = \begin{pmatrix}
				\bH & \bm{0} \\ \star & \bA_{\tilde{I},\lbrack n \rbrack\backslash \lbrack n - k\rbrack}
			\end{pmatrix}.
		\end{align*}
		So we get
		\begin{align*}
			\adj(\bA_{J,\cdot})\be_n = \det \bH \begin{pmatrix}
				\bm{0}\\
				\adj(\bA_{\tilde{I},\lbrack n \rbrack\backslash \lbrack n - k\rbrack})\be_k
			\end{pmatrix}
		\end{align*}
		by using the generalized Laplace expansion along the first $n-k$ rows. Together with the definition of $\br$ this yields
		\begin{align*}
			\br&=\sign(\det \bA_{J,\cdot})\adj(\bA_{J,\cdot})\be_n \\
			&= \sign(\det \bA_{J,\cdot})\det \bH \begin{pmatrix}
				\bm{0}\\
				\adj(\bA_{\tilde{I},\lbrack n \rbrack\backslash \lbrack n - k\rbrack})\be_k \end{pmatrix}\\
			& = \sign(\bA_{\tilde{I},\lbrack n \rbrack\backslash \lbrack n - k\rbrack})|\det \bH| \begin{pmatrix}
				\bm{0}\\
				\adj(\bA_{\tilde{I},\lbrack n \rbrack\backslash \lbrack n - k\rbrack})\be_k \end{pmatrix}.
		\end{align*}
		The normalized generator which corresponds to $\br$ is defined by
		\begin{align*}
			\tilde{\br} =\sign(\bA_{\tilde{I},\lbrack n \rbrack\backslash \lbrack n - k\rbrack})\adj(\bA_{\tilde{I},\lbrack n \rbrack\backslash \lbrack n - k\rbrack})\be_k.
		\end{align*}
		Hence, the last $k$ coordinates of $\br$ are $|\det \bH|\tilde{\br}$.\\
		Solving the full-dimensional version yields non-negative coefficients $\lambda_1,...,\lambda_t$ which sum up to at most one and generate $\tilde{\bh}$ with the respective normalized generators of $\tilde{F}$. We obtain
		\begin{align*}
			\bh = \frac{1}{|\det \bH|}\sum_{i=1}^t\lambda_i\br^i \quad\text{with}\quad\frac{1}{|\det \bH|}\sum_{i=1}^t\lambda_i\leq \frac{1}{|\det \bH|} \leq 1.
		\end{align*}
		It is left to show that we can choose $I\subseteq J$. For each $n-1$ linearly independent rows which define $\br$ and are indexed by $K$ we can partition $K$ into sets $K_1$ and $K_2$ where the rows of $\bA$ indexed by $K_1$ define $F$. The rows of $\bA_{K_1,\cdot}$ are a linear combination of the rows of $(\bH,\bm{0})$. Hence, there exists a matrix $\bQ\in\Q^{(n-k)\times (n-k)}$ with
		\begin{align*}
			\bA_{K_1,\cdot} = \bQ (\bH,\bm{0}).
		\end{align*}
		Recall that $\gcd(\bA_{I,\cdot})=|\det \bH|$, see (\ref{prooffulldimreduction1}). Since $\gcd(\bA_{I,\cdot})$ is by choice maximal among all representations of $F$ with the rows of $\bA$, we have $|\det \bQ|\leq 1$. Applying the generalized Laplace expansion along the rows indexed by $K_1$ yields
		\begin{align*}
			\gcd(\bA_{K,\cdot}) = |\det \bQ| |\det \bH| \gcd(\bA_{K_2,\lbrack n\rbrack \backslash \lbrack n - k\rbrack})=|\det \bQ| \gcd(\bA_{I,\cdot}) \gcd(\bA_{K_2,\lbrack n\rbrack \backslash \lbrack n - k\rbrack}).
		\end{align*}
		As $|\det \bQ|\leq 1$ and $|\det \bQ| = 1$ if $K_1 = I$, we can choose $K_1 = I$ to maximize the expression above. This yields $I\subseteq J$.

	\end{proof}
	
	\subsection{Proof for bimodular cones}
	In this subsection we prove Theorem \ref{bimodtheorem}. In order to do so, we need three lemmas which are valid for the more general class of $\Delta$-modular cones. Let $\bh\in C(\bA)$ be a Hilbert basis element with respect to $\Z^n$. In the following, we assume 
	\begin{align*}
		\bh = \sum_{i=1}^t\lambda_i\br^i
	\end{align*} 
	for non-negative scalars $\lambda_1,...,\lambda_t$ and $\br^1,...,\br^t\in R(\bA)$. If $\bh$ is non-trivial, we already know that $\lambda_i < \frac{1}{\gcd(\br^i)}$ for all $i=1,...,t$, as otherwise we can subtract $\frac{1}{\gcd(\br^i)}\br^i$ from $\bh$ and stay in the cone. The next lemma strengthens this bound by incorporating boundedness of the right-hand side.
	\begin{lemma}\label{bimodlemma1}
		Let $C(\bA)$ be a $\Delta$-modular cone and $\bh$ a non-trivial Hilbert basis element of $C(\bA)$ with respect to $\Z^n$ such that $\bh = \sum_{i=1}^t\lambda_i\br^i$ for non-negative scalars $\lambda_1,...,\lambda_t$ and $\br^1,...,\br^t\in R(\bA)$. Then
		\begin{align*}
			\lambda_i \leq \frac{1}{\gcd(\br^i)} - \frac{1}{\Vert \bA\br^i\Vert_\infty} \leq \frac{1}{\gcd(\br^i)} - \frac{1}{\Delta}
		\end{align*}
		for $i=1,...,t$.
	\end{lemma} 
	\begin{proof}
		Fix some index $i\in\lbrack t\rbrack$ and denote
		\begin{align*}
			\tilde{\br}^i = \frac{1}{\gcd(\br^i)}\br^i\in\Z^n.
		\end{align*}
		We have
		\begin{align*}
			\bA\bh=\sum_{j=1}^t\lambda_j\bA\br^j\geq \lambda_i\bA\br^i= \underbrace{\gcd(\br^i)\lambda_i}_{=:\tilde{\lambda}_i}\bA\tilde{\br}^i .
		\end{align*}
		Since the left side of the inequality is integral, we conclude
		\begin{align*}
			\bA\bh\geq \left\lceil \tilde{\lambda}_i\bA\tilde{\br}^i\right\rceil,
		\end{align*}
		where we apply the ceiling function componentwise. If
		\begin{align*}
			\left\lceil \tilde{\lambda}_i\bA\tilde{\br}^i\right\rceil = \bA\tilde{\br}^i,
		\end{align*}
		we get $\bh-\tilde{\br}^i\in C(\bA)\cap \Z^n$. This yields $\bh = \tilde{\br}^i$, since $\bh$ is a Hilbert basis element and $\tilde{\br}^i\neq\bm{0}$, which contradicts the requirement that $\bh$ is non-trivial. Hence, there exists an index $k\in \lbrack m\rbrack$ such that
		\begin{align*}
			\tilde{\lambda}_i\left(\bA\tilde{\br}^i\right)_k\leq \left(\bA\tilde{\br}^i\right)_k - 1.
		\end{align*}
		After rearranging we get
		\begin{align*}
			\lambda_i = \frac{\tilde{\lambda}_i}{\gcd(\br^i)} \leq \frac{\Vert \bA\tilde{\br}^i\Vert_\infty - 1}{\Vert \bA\tilde{\br}^i\Vert_\infty\gcd(\br^i)} 
			= \frac{1}{\gcd(\br^i)} - \frac{1}{\Vert \bA\br^i\Vert_\infty}.
		\end{align*}
		Applying (\ref{normgenrhs}) proves the second inequality.
	\end{proof}
	An immediate implication of Lemma \ref{bimodlemma1} is that $\br\in R(\bA)$ with $\gcd(\br) = \Delta$ does not contribute to the positive combination if $\bh$ is non-trivial. Specifically, each Hilbert basis element is trivial if $\Delta = 1$, which reproves the statement in that case. Furthermore, all normalized generators which contribute to the positive combination are primitive if $\Delta \in \lbrace 2, 3\rbrace$. That is why, $\gcd(\br)$ divides $\Vert \bA \br\Vert_\infty$ which implies $\gcd(\br) = \Vert \bA\br\Vert_\infty$ if $\Delta\in\lbrace 2, 3\rbrace$. Hence, we do not run into the subtlety which occurred when we defined the normalized generators in the bimodular case.\\
	In the next lemma we investigate the case when $\bh$ is generated by two normalized generators.
	\begin{lemma}\label{bimodlemma2}
		Let $C(\bA)$ be a $\Delta$-modular cone and $\bh$ a Hilbert basis element of $C(\bA)$ with respect to $\Z^n$ which is given by $\bh = \lambda_1 \br^1 + \lambda_2 \br^2$ with $\lambda_1,\lambda_2 >0$ and $\br^1,\br^2\in R(\bA)$. Then
		\begin{align*}
			\lambda_i \geq \frac{1}{\Vert \bA\br^i\Vert_\infty}\geq \frac{1}{\Delta}
		\end{align*}
		for $i=1,2$.
	\end{lemma}
	\begin{proof}
		Choose without loss of generality $i = 1$. Assume, for sake of a contradiction, that $\lambda_1<\frac{1}{\Vert \bA\br^1\Vert_\infty}$. This implies $\left(\lambda_1\bA\br^1\right)_k\notin\Z$ for all $k\in\lbrack m\rbrack$. Hence, 
		\begin{align*}
			\supp(\bA\br^1)\subseteq\supp(\bA\br^2).
		\end{align*}
		Since $\bA \br^2$ is (inclusionwise) minimal, we have equality above which implies $\br^1=\br^2$, a contradiction. Again, the second inequality follows from (\ref{normgenrhs}).
	\end{proof}
	Lastly, we prove a statement about cones with special integer hull.
	\begin{lemma}\label{bimodstrukturklemma}
		Let $C=C(\bA)$ be a cone, $\bv\in \Q^n$ and $\bz\in \bv + C$ the only vertex of the integer hull $\conv((\bv + C)\cap \Z^n)$. Then
		\begin{align*}
			\conv((\bv + C)\cap \Z^n) = \bz + C.
		\end{align*}
	\end{lemma}
	\begin{proof}
		After translation we can assume that $\bz = \bm{0}$.\\
		Since $C$ is defined by an integral matrix, all extreme rays contain integer points. This implies already $\conv((\bv + C)\cap \Z^n) \supseteq C$. Assume there exists
		\begin{align*}
			\bv + \bx \in \conv((\bv + C)\cap \Z^n) \backslash C.
		\end{align*}
		Let $\ba$ be some row of $\bA$ with $\ba^T(\bv+\bx) < 0$. Consider the following linear optimization problem and note that
		\begin{align*}
			\ba^T\bv \leq \min_{\by\in \conv((\bv + C)\cap \Z^n)}\ba^T\by\leq \ba^T(\bv+\bx)  < 0.
		\end{align*}
		This problem is bounded and $\bm{0}$ cannot be a solution. So the optimum is attained at a vertex of $\conv((\bv + C)\cap \Z^n)$ different from $\bz = \bm{0}$. This contradicts our assumption.
	\end{proof}
	In the following proof of Theorem \ref{bimodtheorem} we exploit a result regarding the integer hull of bimodular polytopes proven by Veselov and Chirkov \cite{veselovchirkovbimodular09}.\\
	Here $P$ is a polytope defined by an integral constraint matrix and integral right-hand side and for a vertex $\bv\in P$ we denote by $C^{\bv}$ the vertex cone of $\bv$. Furthermore, an edge is called incident to a vertex if this edge contains the vertex. Below we denote two vertices of a polytope which share an edge as adjacent vertices.
	\begin{theorem}[\cite{veselovchirkovbimodular09}, Theorem 2]
		\label{bimodveselovchirkov}
		Let $P(\bA,\bb)$ be a full-dimensional bimodular polytope and $\bv$ a vertex of $P(\bA,\bb)$. Then each vertex of $\conv(C^{\bv}\cap\Z^n)$ lies on an edge of $P(\bA,\bb)$ incident to $\bv$.
	\end{theorem}
	We are now in the position to prove our main result.
	\begin{proof}[Proof of Theorem \ref{bimodtheorem}]
		Firstly, we assume that $\bh$ is a positive combination of precisely two normalized generators. By Lemma \ref{bimodlemma1} and Lemma \ref{bimodlemma2} the two coefficients need to be $\frac{1}{2}$. It is left to show that every $\bh$ can be generated by at most two normalized generators.\\
		Recall from (\ref{spindleproperty}) that the spindle $S(\bh)$ satisfies $S(\bh)\cap \Z^n=\lbrace \bm{0}, \bh\rbrace$. Further, we can assume
		\begin{align}\label{proofbimodinterior}
			\bh \in\intt(C(\bA))
		\end{align}
		by Proposition \ref{pointinterior}.
		As $\bh$ is non-trivial and lies in the interior of $C(\bA)$, the spindle $S(\bh)$ is a full-dimensional bimodular polytope with $\dim(S(\bh))=n\geq 2$. Therefore, there exists a vertex $\bv\in S(\bh)$ which is neither $\bm{0}$ nor $\bh$. We claim that $\bv$ is adjacent to $\bm{0}$ and $\bh$. This implies that $\bv$ and $\bh-\bv$ lie both on extreme rays of $C(\bA)$ by symmetry of $S(\bh)$. So we have $\bh = (\bh -\bv) + \bv$, i.e., $\bh$ is the $\frac{1}{2}$ combination of precisely two normalized generators, $2(\bh-\bv)$ and $2\bv$. This completes the proof.\\ 
		It remains to show that $\bv$ has to be adjacent to $\bm{0}$ and $\bh$. Assume this is not the case. Let $C^{\bv}$ be the vertex cone at $\bv$. The integer hull $\conv(C^{\bv}\cap \Z^n)$ contains $\bm{0}$ and $\bh$. From Theorem \ref{bimodveselovchirkov} it follows that the vertices of $\conv(C^{\bv}\cap \Z^n)$ are contained in $S(\bh)\cap \Z^n =\lbrace \bm{0},\bh\rbrace$. Hence, $\conv(C^{\bv}\cap \Z^n)$ has at most two vertices, $\bm{0}$ and $\bh$. Further, the vertices of $\conv(C^{\bv}\cap \Z^n)$ lie on edges of $S(\bh)$ incident to $\bv$ by Theorem \ref{bimodveselovchirkov}. This combined with our assumption that $\bv$ is not adjacent to $\bm{0}$ and $\bh$ implies that $\conv(C^{\bv}\cap \Z^n)$ contains exactly one vertex. Without loss of generality let this be $\bm{0}$. Using Lemma \ref{bimodstrukturklemma} we observe
		\begin{align}
			\label{proofbimodintegerhull}
			\conv(C^{\bv}\cap \Z^n) = C^{\bv}-\bv.
		\end{align}
		Since  $\bv \neq \bm{0}$, there exists a row $\ba$ of $\bA$ with $\ba^T\bv = \ba^T\bh> 0$, where we use (\ref{proofbimodinterior}) for the strict inequality. Combining this with (\ref{proofbimodintegerhull}) we get that $\ba^T\bx \leq 0$ is an inequality of the integer hull $\conv(C^{\bv}\cap \Z^n)$. However, $\bh\in \conv(C^{\bv}\cap \Z^n)$ implies that
		\begin{align*}
			0\geq \ba^T\bh > 0.
		\end{align*}
		This is a contradiction. Hence, every non-integral vertex of $S(\bh)$ is adjacent to $\bm{0}$ and $\bh$.
	\end{proof}
	Note that the proof of Theorem \ref{bimodtheorem} shows us that the diameter from $\bm{0}$ to $\bh$ of the vertex-edge graph induced by $S(\bh)$ is equal to two. Moreover, let $V$ be the number of vertices of $S(\bh)$. Then there are $\frac{V - 2}{2}$ possibilities to express $\bh$ as a positive combination of two normalized generators.
	
	\subsection{Proof for simplicial cones} \label{ssectionsimplicial}
	\begin{proof}[Proof of Proposition \ref{transformationprop}]
		We start by proving the first claim. Integrality of the matrices implies $\bA\bB\in\Z^{m\times n}$. We set $\Lambda = \bB\Z^n$. Let $\bB^{-1}\bh\in C(\bA\bB) \cap \Z^n$ be a Hilbert basis element of $ C(\bA\bB)$ with respect to $\Z^n$. Then $\bh\in C(\bA) \cap \Lambda$ is a Hilbert basis element of $C(\bA)$ with respect to $\Lambda$.\\
		Since $\Lambda \subseteq \Z^n$, we have $\bh\in \Z^n$. Hence, we can express $\bh$ as a non-negative integral sum of Hilbert basis elements $\tilde{\bh}^1,...,\tilde{\bh}^s$ with respect to $\Z^n$. So there are $\lambda_1,...,\lambda_t\in\Z_{\geq 0}$ with
		\begin{align}
			\label{prooftransformation1}
			\bh = \sum_{i = 1}^s\lambda_i\tilde{\bh}^i = \underbrace{\tilde{\bh}^1 + ... +\tilde{\bh}^1}_{\lambda_1\text{-times}}+ ... + \underbrace{\tilde{\bh}^s + ... + \tilde{\bh}^s}_{\lambda_s\text{-times}}.
		\end{align} 
		First, we aim to bound $\lambda_1+...+\lambda_s =: \lambda$. For that purpose, we denote by $\by^1,...,\by^{\lambda}$ the terms in (\ref{prooftransformation1}). We observe that
		\begin{align*}
			\by^1,\by^1+\by^2,...,\by^1+...+\by^{\lambda - 1}\in (S(\bh)\cap \Z^n)\backslash\lbrace \bm{0},\bh\rbrace
		\end{align*}
		as each subsum of $\by^1+...+\by^{\lambda}$ satisfies $\bm{0}\leq \bA\bx\leq \bA\bh$.\\
		Since $\bh$ is Hilbert basis element of $C(\bA)$ with respect to $\Lambda$, we get $S(\bh)\cap \Lambda =\lbrace \bm{0},\bh\rbrace$ from (\ref{spindleproperty}). Therefore, none of the sums above can be in $\Lambda$. So they are non-trivial elements in the quotient group $\Z^n / \Lambda$. The index of $\Z^n / \Lambda$ is $|\det \bB|$. If $\lambda\geq |\det \bB| + 1$, we have two sums, say $\by^1+...+\by^q$ and $\by^1+...+\by^p$ with $1\leq q < p \leq \lambda - 1$, which lie in the same coset. This implies
		\begin{align*}
			\by^{q+1}+...+\by^p = (\by^1+...+\by^p) - (\by^1+...+\by^q)\in S(\bh)\cap \Lambda
		\end{align*}
		and thus contradicts $S(\bh)\cap \Lambda =\lbrace \bm{0},\bh\rbrace$. Hence,
		\begin{align}
			\label{prooftransformation2}
			\lambda_1+...+\lambda_s = \lambda\leq |\det \bB|.
		\end{align}
		Second, we apply our assumption that (SHC) holds for $C(\bA)$. Therefore, each Hilbert basis element in (\ref{prooftransformation1}) is a convex combination of the normalized generators of $C(\bA)$. Let $R(\bA) = \lbrace\br^1,...,\br^t\rbrace$. As (SHC) holds for $C(\bA)$ and with (\ref{prooftransformation1}) and (\ref{prooftransformation2}), we obtain
		\begin{align*}
			\bh \in |\det \bB|\conv\lbrace \bm{0},\br^1,...,\br^t\rbrace.
		\end{align*}
		Transforming back with $\bB^{-1}$ results in 
		\begin{align}
			\label{prooftransformation3}
			\bB^{-1}\bh \in \conv\lbrace \bm{0},|\det \bB| \bB^{-1} \br^1,...,|\det \bB| \bB^{-1} \br^t\rbrace.
		\end{align}
		By Cramer's rule we have $|\det \bB| \bB^{-1} \br^1,...,|\det \bB| \bB^{-1} \br^t\in\Z^n$. It suffices to check that these vectors are the normalized generators of $C(\bA\bB)$ or equivalently $R(\bA\bB) = |\det \bB|\bB^{-1}R(\bA)$. Let $I\subseteq\lbrack m\rbrack$ denote the indices which correspond to the $(n - 1)\times n$ submatrix $\bA_{I,\cdot}$ of $\bA$ in the definition of the  generator $\br^1$. Then we have
		\begin{align*}
			\bm{0} = \bA_{I,\cdot}\br^1 =\bA_{I,\cdot}\bB(\bB^{-1}\br^1) = (\bA\bB)_{I,\cdot}(\bB^{-1}\br^1). 
		\end{align*}
		Assume $\br^1$ is given as in (\ref{alternativedefnormgen}). Thus,
		\begin{align*}
			|\det \bB| \bB^{-1} \br^1 = |\det \bA_{I\cup \lbrace j \rbrace,\cdot}\bB| (\bA_{I\cup \lbrace j \rbrace,\cdot}\bB)^{-1}\be_n.
		\end{align*}
		Hence, $|\det \bB| \bB^{-1} \br^1$ is potentially a normalized generator. We still need to verify that the scaling is correct. Let the index set $K\subseteq\lbrack m\rbrack$ denote another matrix with $\bm{0}=\bA_{K,\cdot}\br^1$ and $\rank(\bA_{K,\cdot}) = n -1$. Since the kernels of $\bA_{I,\cdot}$ and $\bA_{K,\cdot}$ are equal, there exists $\bQ\in\Q^{(n-1)\times (n-1)}$ with  
		\begin{align*}
			\bQ\bA_{I,\cdot} = \bA_{K,\cdot}
		\end{align*}
		and $|\det \bQ|\leq 1$ by maximality of $\gcd(\bA_{I,\cdot})$. Therefore,
		\begin{align*}
			\gcd((\bA\bB)_{K,\cdot}) = |\det \bQ| \gcd((\bA\bB)_{I,\cdot})\leq \gcd((\bA\bB)_{I,\cdot})
		\end{align*}
		shows that $|\det \bB| \bB^{-1} \br^1\in R(\bA\bB)$ and in general $R(\bA\bB) = |\det \bB|\bB^{-1}R(\bA)$.\\
		For the second claim we assume that $\bB^T$ is a basis of the lattice $\bA^T\Z^m$. Thus, we can decompose $\bA = \tilde{\bA}\bB$ where $\tilde{\bA}\in\Z^{m\times n}$ since each row of $\bA$ is in $\bB^T\Z^n$. As a result, we obtain $C(\bA)= C(\tilde{\bA}\bB)$. From the first claim it suffices to show (SHC) for $C(\tilde{\bA})$ which proves the statement since $\gcd(\tilde{\bA}) = 1$, see Lemma \ref{latticedeterminant}.
	\end{proof}
	
	\begin{proof}[Proof of Corollary \ref{simplicialstatement}]
		The simplicial cone $C(\bA)$ is given by an invertible matrix $\bA\in\Z^{n\times n}$. Proposition \ref{transformationprop} implies that we can assume $1 = \gcd(\bA) = |\det \bA|$. Hence, the first claim follows immediately from the case $\Delta(\bA) = 1$, e.g., compare with \cite[Proposition 8.1]{Sturmfels1995GrobnerBA}.\\
		In order to prove the second claim, we look closer into the proof of Proposition \ref{transformationprop}. Note that $\bh$ lies in the relative interior of a $k$-face of $C(\bA)$ if and only if $\bA\bh$ lies in the relative interior of a $k$-face of $C(\bm{I}_n)$. So we are in the special case where the constraint matrix is the $n \times n$ unit matrix. Further, the Hilbert basis elements of $C(\bm{I}_n)$ are the unit vectors $\be_1,...,\be_n$. That and (\ref{prooftransformation1}) imply
		\begin{align*}
			\bA\bh = \sum_{i = 1}^n\lambda_i\be_i.
		\end{align*}
		Moreover, the number of non-zero entries of $\bA\bh$ corresponds to the dimension of the face which contains $\bA\bh$ in its interior as $C(\bm{I}_n)$ is simplicial. In other words, $|\supp(\bA\bh)|$ is the dimension we are trying to bound. Putting everything together yields
		\begin{align*}
			|\supp(\bA\bh)| \leq \lambda_1+...+\lambda_n\leq |\det \bA| = \Delta
		\end{align*}
		where we used (\ref{prooftransformation2}) for the second inequality.
	\end{proof}
	
	\subsection{Proof of Theorem \ref{theoremsumupSHC}} \label{ssectionsumupproof}
	\begin{proof}[Proof of Theorem \ref{theoremsumupSHC}]
		The first statement is a consequence of Corollary \ref{simplicialstatement} and \cite{papageometryminimalsolution21}. The case $k = 1$ for the second statement follows directly from the result for $\Delta(\bA) = 1$ and Theorem \ref{bimodtheorem}. For arbitrary $k\in \N_{\geq 2}$ observe that $\gcd(\bA) \geq k$. We can modify the cone such that $k=1$ as in the proof of Proposition \ref{transformationprop}.
	\end{proof}
	
	\bibliographystyle{plain}
	
	\bibliography{references}
	
	\appendix
	
	\section{Appendix}
	\subsection{Example: Maximal scaling is necessary} \label{appendixexample1}
	This explicit example is a product of computational experiments on a computer.\\
	Given the following constraint matrix
	\begin{align*}
		\bA = \begin{pmatrix}
			1 & 0 & 0 & 0 \\
			0 & 1 & 0 & 0 \\
			1 & 8 & 4 & 11 \\
			1 & 4 & 3 & 6 \\
			-1 & -7 & -3 & -10 \\
			-1 & -7 & -2 & -9
		\end{pmatrix}
	\end{align*}
	and the respective cone $C(\bA)$. The normalized generators of $C(\bA)$ are the columns of the matrix below
	\begin{align*}
		\begin{pmatrix}
			14 & 0 & 0 & 0 & 0 & 1 & 0 \\
			0 & 7 & 0 &15 & 9& 8 & 0\\
			2 & 7 & 10 & 6 & 4 & 3 & 11\\
			-2 & -7 & -3 & -13 & -8 & -7 & -4
		\end{pmatrix} = (\br^1,...,\br^7).
	\end{align*}
	We have 
	\begin{align*}
		\begin{pmatrix}
			0 & 1 & 0 & 0\\
			1 & 8 & 4 & 11 \\
			-1 & -7 & -3 & -10\\
			-1 & -7 & -2 & -9
		\end{pmatrix}\br^1 = \bm{0}.
	\end{align*}
	Note that $I\subseteq\lbrace 2,3,5,6\rbrace$ with $I=\lbrace 2,3,6\rbrace$ gives us $\gcd(\bA_{I,\cdot})=2$ whereas all other three element subsets of $\lbrace 2,3,5,6\rbrace$ yield a gcd of 1. If we would pick one of those, we get the vector $\frac{1}{2}\br^1$. However, the Hilbert basis element 
	\begin{align*}
		\bh = \begin{pmatrix}
			6 \\ 1 \\ 2 \\ -2
		\end{pmatrix}
	\end{align*}
	of $C(\bA)$ with respect to $\Z^4$ can be expressed by
	\begin{align*}
		\bh = \frac{3}{7}\br^1+\frac{1}{9}\br^5+\frac{4}{63}\br^7=\frac{47}{112}\br^1+\frac{1}{8}\br^6+\frac{1}{14}\br^7
	\end{align*}
	which yields coordinate sums greater than 1 if we replace $\br^1$ with $\frac{1}{2}\br^1$.
	
\end{document}